\documentclass[a4paper,oneside]{amsart}
\usepackage{geometry}

\usepackage{mathtext}
\usepackage[T1]{fontenc}
\usepackage{color}
\usepackage{amsmath}
\usepackage{amsfonts,amssymb,amsthm}
\usepackage[pdftex]{graphicx}
\usepackage{comment}
\usepackage{float}
\usepackage{todonotes}
\usepackage{microtype}

\usepackage[colorlinks=false, linktocpage=true]{hyperref}

\usepackage{tikz}
\usetikzlibrary{positioning, arrows.meta}
\usetikzlibrary{decorations.markings}
\tikzset{middlearrow/.style={
        decoration={markings,
            mark= at position 0.5 with {\arrow{#1}} ,
        },
        postaction={decorate}
    }
}
\tikzset{firstthirdarrow/.style={
        decoration={markings,
            mark= at position 0.33 with {\arrow{#1}} ,
        },
        postaction={decorate}
    }
}
\tikzset{secondthirdarrow/.style={
        decoration={markings,
            mark= at position 0.66 with {\arrow{#1}} ,
        },
        postaction={decorate}
    }
}
\usetikzlibrary{cd}

\usepackage{enumerate}

\usepackage{cite}

\usepackage{color}
\definecolor{NoteColor}{rgb}{1,0,0}

\newcommand{\midwd}{\,\middle\vert\,}

\DeclareMathOperator{\rk}{rk}

\DeclareMathOperator{\Aut}{Aut}

\DeclareMathOperator{\Hom}{Hom}
\DeclareMathOperator{\diag}{diag}
\DeclareMathOperator{\Id}{Id}
\DeclareMathOperator{\Sym}{Sym}

\DeclareMathOperator{\Mat}{Mat}

\DeclareMathOperator{\Sp}{Sp}
\DeclareMathOperator{\GL}{GL}
\DeclareMathOperator{\SL}{SL}
\DeclareMathOperator{\PSL}{PSL}

\DeclareMathOperator{\OO}{O}
\DeclareMathOperator{\UU}{U}

\DeclareMathOperator{\SO}{SO}

\DeclareMathOperator{\Rep}{Rep}
\DeclareMathOperator{\Stab}{Stab}

\DeclareMathOperator{\spp}{\mathfrak{sp}}

\DeclareMathOperator{\End}{End}

\DeclareMathOperator{\Spin}{Spin}

\DeclareMathOperator{\Fix}{Fix}
\DeclareMathOperator{\KSp}{KSp}
\DeclareMathOperator{\Herm}{Herm}

\DeclareMathOperator{\Is}{Is}

\DeclareMathOperator{\fr}{fr}

\DeclareMathOperator{\SU}{SU}

\newcommand{\R}{\mathbb R}

\newcommand{\CC}{\mathbb C}
\newcommand{\HH}{\mathbb H}

\newcommand{\PP}{\mathbb P}
\newcommand{\N}{\mathbb N}

\newcommand{\K}{\mathcal K}
\newcommand{\Oc}{\mathbb O}
\newcommand{\T}{\mathcal T}
\newcommand{\Pp}{\mathcal P^+}
\newcommand{\Pm}{\mathcal P^-}

\newcommand{\F}{\mathcal F}
\renewcommand{\L}{\mathcal L}
\renewcommand{\T}{\mathcal T}

\newcommand{\coloneqq}{\mathrel{\mathop:}=}

\newcommand{\Ome}[1]
{\begin{pmatrix}
0 & #1\\
-#1 & 0
\end{pmatrix}}

\newcommand{\Om}{
\begin{pmatrix}
0 & 1\\
-1 & 0
\end{pmatrix}}

\newcommand{\Clm}[2]
{\begin{pmatrix}
#1\\
#2
\end{pmatrix}}

\DeclareMathOperator{\PIs}{\PP(\Is(\omega))}

\theoremstyle{plain}
\newtheorem{teo}{Theorem}[section]

\newtheorem{cor}[teo]{Corollary}

\newtheorem{lem}[teo]{Lemma}

\newtheorem{prop}[teo]{Proposition}

\newtheorem{fact}[teo]{Fact}

\theoremstyle{definition}
\newtheorem{df}[teo]{Definition}

\theoremstyle{remark}
\newtheorem{rem}[teo]{Remark}

\newcommand{\bs}{\smallsetminus}
\newcommand{\defin}{\emph}

\renewcommand{\emptyset}{\varnothing}

\begin{document}

\sloppy

\title{On a parametrization of spaces of maximal framed representations}

\author[Eugen Rogozinnikov]{Eugen Rogozinnikov} \address{Institut de Recherche Math\'ematique Avanc\'ee, Universit\'e de Strasbourg, 7 rue Ren\'e Descartes, 67000 Strasbourg, France}
\email{erogozinnikov@gmail.com}

\begin{abstract}
    This article is an extended version of the talk given by the author in the seminar Théorie Spectrale et Géométrie at the Institut Fourier in March 2022. We present some results of the PhD-thesis of the author~\cite{R-Thesis} extended by several results from the articles~\cite{AGRW, ABRRW, GRW}. We give a parametrization of the space of maximal framed representations of the fundamental group of a punctured surface into a Hermitian Lie group of tube type that can be seen as $\Sp_2(A,\sigma)$ for a Hermitian algebra $(A,\sigma)$. Using this parametrization, we count connected components of the space of maximal framed representations as well as the space of maximal (non-framed) representations.
\end{abstract}

\maketitle

\tableofcontents
\newpage 
\section{Introduction}

Higher Teichm\"uller theory  was developed as a generalization of classical Teichm\"uller theory that studies moduli spaces of complex structures on a fixed topological surface $S$ of negative Euler characteristic. This moduli space is called \emph{Teichm\"uller space} $\mathcal T(S)$, and it  can also be seen as the moduli space of marked complete hyperbolic structures on the surface $S$. Teichm\"uller space $\mathcal T(S)$ can be naturally embedded into the representation variety $\Hom(\pi_1(S),\PSL(2,\R))/\PSL(2,\R)$ as a connected component which consists entirely of discrete and faithful representations.

Higher Teichm\"uller theory generalizes this approach and studies representations of $\pi_1(S)$ into a reductive Lie group $G$ of higher rank. A \emph{higher Teichm\"uller space} is a subset of $\Rep(\pi_1(S),G):=\Hom(\pi_1(S),G)/G$ which is a union of connected components that consist entirely of discrete and faithful representations. There are two well-known families of Higher Teichm\"uller spaces: Hitchin components and spaces of maximal representations.

Hitchin components are defined when $G$ is a split real simple Lie group (e.g. $\SL(n,\R)$). The structure of Hitchin components is well-studied for closed surfaces~\cite{H92} as well as for surfaces with punctures and boundary components~\cite{FG}. In particular, their topology is well understood: every Hitchin component is homeomorphic to an open ball in a Euclidean vector space. Especially, they are contractible manifolds.

Maximal representations were introduced and studied in~\cite{BIW,BILW,Strubel}. The topology of spaces of maximal representation for closed surfaces was studied in~\cite{AC,CTT,Gothen} using the theory of Higgs bundles. In~\cite{AGRW}, the spaces of framed and decorated maximal representations of the fundamental group of a punctured surface into the real symplectic group $\Sp(2n,\R)$ are parametrized using a noncommutative analog of the Fock--Goncharov parametrization~\cite{FG} and the topology of them is studied. In~\cite{KR} maximal representations for punctured surfaces are studied from the point of view of the theory of spectral networks, the authors show that for maximal framed representations the partial abelianization procedure can be applied. In contrast to the Hitchin components, the topology of spaces of maximal representations is much more complicated. Nevertheless, as described before, Hitchin components and spaces maximal representations share many properties~\cite{BILW,L06}, and moreover, if $G=\PSL(2,\R)$, agree and coincide with the Teichm\"uller space $\mathcal T(S)$ \cite{WienhardICM}.


However, higher Teichm\"uller spaces do not exist for every Lie group $G$. Discrete and faithful representations form in general only a closed subset of $\Rep(\pi_1(S),G)$ but not connected components.

A new example of Lie groups which admit higher Teichm\"uller spaces is conjectured to be Lie groups with a notion of positivity. The theory of $\Theta$-positivity was developed by O.\;Guichard and A.\;Wienhard and generalizes Lusztig's total positivity for split real Lie groups and maximality for Hermitian Lie groups to a larger class of simple Lie groups (e.g. $\SO(p,q)$, $p\neq q$)~\cite{GW1,GLW,GW2}.

In this article we concentrate our attention on Hermitian Lie groups of tube type and maximal representations. Following~\cite{ABRRW} we introduce symplectic groups $\Sp_2(A,\sigma)$ for unital associative possibly noncommutative real algebras $A$ with an anti-involution $\sigma$, i.e. with a $\R$-linear map $\sigma\colon A\to A$ such that $\sigma(ab)=\sigma(b)\sigma(a)$ for all $a,b\in A$ and $\sigma^2=\Id$. We show that the for a special class of algebras, called Hermitian algebras, the group $\Sp_2(A,\sigma)$ is always Hermitian of tube type.

Further, following~\cite{AGRW} for an orientable punctured surface $S$ of negative Euler characteristic, we define the space $\Rep^{\fr}(\pi_1(S),G)$ of framed representations of $\pi_1(S)$ into $\Sp_2(A,\sigma)$, i.e. representation equipped with some additional structure around punctures called framing. Inside the space of framed representations we define the space of maximal framed representations $\Rep_+^{\fr}(\pi_1(S),G)$. We parametrize the space of maximal frames representation using ideal triangulations of $S$, i.e. triangulations of $S$ whose set of vertices agrees with the set of punctures of $S$. The following theorem is the main result of this article:

\begin{teo}\label{intro:main_thm}
Let $S$ be an orientable punctured surface $S$ without boundary of negative Euler characteristic $\chi(S)$. The space $\Rep_+^{\fr}(\pi_1(S),G)$ is homeomorphic to the following space:
$$((A^\sigma_+)^{-3\chi(S)}\times U_{(A,\sigma)}^{1-\chi(S)})/U_{(A,\sigma)}$$
where $A^\sigma_+=\{a^2\mid a\in A^\times,\;\sigma(a)=a\}$, $U_{(A,\sigma)}=\{a\in A^\times\mid \sigma(a)=a^{-1}\}$ and $U_{(A,\sigma)}$ acts by conjugation on every factor.
\end{teo}

Using this result, we count connected components of spaces of maximal representations for classical Hermitian Lie groups of tube type.

{\bf Structure of the paper:}
In Section~\ref{sec:symp_gr} we introduce involutive associative algebras and Hermitian algebras $(A,\sigma)$ and define the symplectic group $\Sp_2(A,\sigma)$.
In Section~\ref{sec:is_islines} we introduce the space of isotropic lines and study the action of $\Sp_2(A,\sigma)$ on it.
In Section~\ref{sec:top_data} we introduce punctured surfaces, ideal triangulations and some special kind of graphs on surfaces associated to triangulations that we are using later.
In Section~\ref{sec:framed_reps}, we introduce the spaces of framed representations and local systems on graphs and describe the connection between them.
In Section~\ref{section:parametrization} we prove Theorem~\ref{intro:main_thm}, apply it to understand the homotopy type of the space of maximal framed representations and count its connected components.

{\bf Acknowledgements.} I would like to express my gratitude to my research advisers Anna~Wienhard and Daniele~Alessandrini, and to my collaborators Arkady~Berenstein, Olivier~Guichard, and Vladimir~Retakh. This article is essentially based on our common papers~\cite{AGRW,ABRRW,GRW}. I thank Vladimir~Fock, Michael~Gekhtman, Michael~Shapiro and Marta Magnani for helpful and interesting discussions about some aspects of this article. The author acknowledges funding by the Deutsche Forschungsgemeinschaft within the RTG 2229 ``Asymptotic invariants and limits of groups and spaces'' and the Priority Program SPP 2026 ``Geometry at Infinity'', and thanks the Labex IRMIA of the Universit\'e de Strasbourg for support during the preparation of this article.

\section{Symplectic groups over noncommutative algebras}\label{sec:symp_gr}

\subsection{Involutive algebras and Hermitian algebras} In this section we recall the definition of a Hermitian algebra and state some properties of it that will be important for us later. For more detailed discussion of involutive and Hermitian algebras, we refer to~\cite[Section~2]{ABRRW}.

Let $A$ be a unital associative possibly noncommutative finite dimensional semisimple $\R$-algebra.

\begin{df}\label{def:antiinv}
An \defin{anti-involution} on $A$ is a $\R$-linear map $\sigma\colon A\to A$ such that
\begin{itemize}
\item $\sigma(ab)=\sigma(b)\sigma(a)$;
\item $\sigma^2=\Id$.
\end{itemize}
An \defin{involutive $\R$-algebra} is a pair $(A,\sigma)$, where $A$ is a $\R$-algebra as above and $\sigma$ is an anti-involution on $A$.
\end{df}

\begin{rem}
Sometimes in the literature the maps that satisfy  Definition~\ref{def:antiinv} are called just involutions. We add the prefix ``anti'' in order to emphasise that they exchange the factors.
\end{rem}

\begin{df} An element $a\in A$ is called \defin{$\sigma$-normal} if $\sigma(a)a=a\sigma(a)$. An element $a\in A$ is called \defin{$\sigma$-symmetric} if $\sigma(a)=a$. An element $a\in A$ is called \defin{$\sigma$-anti-symmetric} if $\sigma(a)=-a$. We~denote:
$$A^{\sigma}:=\Fix_A(\sigma)=\{a\in A\mid \sigma(a)=a\},$$
$$A^{-\sigma}:=\Fix_A(-\sigma)=\{a\in A\mid \sigma(a)=-a\}.$$
\end{df}

\begin{rem}
Since the algebra $A$ is unital, we always have the canonical copy of $\R$ in $A$, namely $\R\cdot 1$ where $1$ is the unit of $A$. We always identify $\R\cdot 1$ with $\R$. Moreover, since $\sigma$ is linear, for all $k\in \R$, $\sigma(k\cdot 1)=k\sigma(1)=k\cdot 1$, i.e. $\R\cdot 1\subseteq A^\sigma$. Involutive algebras with $A^\sigma=\R\cdot 1$ are called \defin{thin}.
\end{rem}

We denote by $A^\times$ the group of all invertible elements of $A$. If $V \subset A$ is a vector subspace, we denote $V^\times = A^\times \cap V$
the set of invertible elements in $V$.

\begin{df}
The closed subgroup $$U_{(A,\sigma)}=\{a\in A^\times\mid \sigma(a)=a^{-1}\}$$ of $A^\times$ is called the \defin{unitary group} of $A$. The Lie algebra of $U_{(A,\sigma)}$ agrees with $A^{-\sigma}$.
\end{df}

\begin{df}
\label{df:cone}
Let $(A,\sigma)$ be an involutive algebra. We define the set of $\sigma$-positive elements by
$$A^\sigma_+:=\left\{ a^2  \mid a\in (A^\sigma)^\times\right\},$$
and the set of $\sigma$-non-negative elements by
$$A^\sigma_{\geq 0}:=\left\{ a^2 \mid a\in A^\sigma\right\}.$$
\end{df}

The following definition of a Hermitian algebra will be one of key notions in this paper.

\begin{df}\label{Herm_A}
An involutive algebra $(A,\sigma)$ is called \defin{Hermitian} if for all $x,y\in A^\sigma$, $x^2+y^2=0$ implies $x=y=0$.
\end{df}

For further considerations the following definition will be important:
\begin{df}
Let $V$ be an $\R$-vector space. A subset $\Omega\subseteq V$ is called a \defin{cone} in $V$ if $\lambda x\in \Omega$ for all $\lambda\in (0,\infty)$, $x\in\Omega$. A cone $\Omega$ is called \defin{convex} if $x+y\in\Omega$ for all $x,y\in\Omega$. A cone $\Omega$ is called \defin{proper} if it does not contain affine lines.
\end{df}

\begin{rem}[{\cite[Section~2]{ABRRW}}]
Let $(A,\sigma)$ be a Hermitian algebra. The spaces $A^\sigma_{\geq 0}$ and $A^\sigma_+$ are proper convex cones in $A^\sigma$. The space $A^\sigma_{\geq 0}$ is the topological closure of $A^\sigma_+$. The space $A^\sigma_+$ is the open core of $A^\sigma_{\geq 0}$ in $A^\sigma$. The group $U_{(A,\sigma)}$ is compact, and it acts on $A^\sigma_+$ and on $A^\sigma_{\geq 0}$ by conjugation.
\end{rem}

\begin{df}
Let $x\in A$ and $a\in A^\times$. The element $ax\sigma(a)$ is called \defin{congruent} to $x$ by $a$. The group $A^\times$ acts on $A$ by \defin{congruence} preserving $A^\sigma$ and $A^\times$.
\end{df}

Notice that the restriction of the action by congruence to $U_{(A,\sigma)}<A^\times$ agrees with the action by conjugation of $U_{(A,\sigma)}$.

\begin{rem}[{\cite[Corollary~2.65]{ABRRW}}]
Let $(A,\sigma)$ be a Hermitian algebra. The action of $A^\times$ by congruence preserves $A^\sigma_{\geq 0}$ and $A^\sigma_{+}$.
\end{rem}

The following involutive algebras provide examples of Hermitian algebras:
\begin{enumerate}
\item Let $A=\Mat(n,\R)$ be the space of real $n\times n$-matrices and $\sigma(r):=r^T$ for $r\in A$. Then $A^\sigma=\Sym(n,\R)$ is the space of symmetric matrices. The algebra $(A,\sigma)$ is Hermitian with $A^{\sigma}_+=\Sym^+(n,\R)$ real symmetric positive definite matrices.

\item Let $A=\Mat(n,\CC)$ be the space of complex $n\times n$-matrices and $\bar\sigma(r):=\bar r^T$ for $r\in A$. Then $A^{\bar\sigma}=\Herm(n,\CC)$ is the space of complex Hermitian matrices. The algebra $(A,\bar\sigma)$ is Hermitian with $A^{\bar\sigma}_+=\Herm^+(n,\CC)$ complex Hermitian positive definite matrices.

\item Let $A=\Mat(n,\HH)$ be the space of quaternionic $n\times n$-matrices and $\bar\sigma(r):=\bar r^T$ for $r\in A$. Then $A^{\bar\sigma}=\Herm(n,\HH)$ is the space of quaternionic Hermitian matrices. The algebra $(A,\bar\sigma)$ is Hermitian with $A^{\bar\sigma}_+=\Herm^+(n,\HH)$ quaternionic Hermitian positive definite matrices.
\end{enumerate}

\subsection{Sesquilinear forms on \texorpdfstring{$A$}{A}-modules and their groups of symmetries}

Let $(A,\sigma)$ be an involutive algebra.

\begin{df}\label{osp}
A \defin{$\sigma$-sesquilinear form} $\omega$ on a right $A$-module $V$ is a map
$$\omega\colon V\times V\to A$$
such that for all $x,y,z\in V$ and for all $r_1,r_2\in A$
$$\omega(x+y,z)=\omega(x,z)+\omega(y,z),$$
$$\omega(x,y+z)=\omega(x,y)+\omega(x,z),$$
$$\omega(x_1r_1,x_2r_2)=\sigma(r_1)\omega(x_1,x_2)r_2.$$

We denote by $$\Aut(\omega):=\{f\in\Aut(V)\mid \forall x,y\in V:\omega(f(x),f(y))=\omega(x,y)\}$$ the group of symmetries of $\omega$. We also define the corresponding Lie algebra:
$$\End(\omega):=\{f\in\End(V)\mid \forall x,y\in V:\omega(f(x),y)+\omega(x,f(y))=0\}$$
with the usual Lie bracket $[f,g]=fg-gf$.
\end{df}

We now set $V=A^2$. We view $V$ as the set of columns and endow it with the structure of a right $A$-module. The anti-involution $\sigma$ extends to an involutive map on $A^2$ and on the space $\Mat_2(A)$ of $2\times2$-matrices with coefficients in $A$ componentwisely.

\begin{df}$ $
We make the following definitions:
\begin{enumerate}
\item A pair $(x,y)$ for $x,y\in A^2$ is called a \defin{basis} of $A^2$ if for every $z\in A^2$ there exist $a,b\in A$ such that $z=xa+yb$.
\item The element $x\in A^2$ is called \defin{regular} if there exists $y\in A^2$ such that $(x,y)$ is a basis of $A^2$.
\item $l\subseteq A^2$ is called a \defin{line} if $l=xA$ for a regular $x\in A^2$. We denote the space of lines of $A^2$ by $\PP(A^2)$.
\item Two regular elements $x,y\in A^2$ are called \defin{linearly independent} 
if $(x,y)$ is a basis of $A^2$.
\item Two lines $l,m$ are called \defin{transverse} if $l=xA$, $m=yA$ for linearly independent $x,y\in A^2$.
\item An element $x\in A^2$ is called \defin{isotropic} with respect to $\omega$ if $\omega(x,x)=0$. The set of all isotropic regular elements of $(A^2,\omega)$ is denoted by $\Is(\omega)$.
\item A line $l$ is called isotropic if $l=xA$ for a regular isotropic $x\in A^2$. The set of all isotropic lines of $(A^2,\omega)$ is denoted by $\PP(\Is(\omega))$.
\end{enumerate}
\end{df}

From now on, we assume $\omega(x,y):=\sigma(x)^T\Omega y$ with $\Omega=\Ome{1}$, $x,y\in A^2$. The form $\omega$ is called the \defin{standard symplectic form} on $A^2$. We introduce the symplectic group $\Sp_2(A,\sigma)$ over $(A,\sigma)$.

\begin{df}
The group $\Sp_2(A,\sigma):=\Aut(\omega)$ is the \defin{symplectic group} $\Sp_2$ over $(A,\sigma)$. $\Sp_2(A,\sigma)$ is a Lie group, its Lie algebra is $\spp_2(A,\sigma):=\End(\omega)$.
\end{df}

We have
$$\Sp_2(A,\sigma)=\left\{\begin{pmatrix}
    a & b \\
    c & d
    \end{pmatrix}\midwd \sigma(a)c,\,\sigma(b)d\in A^\sigma,\,\sigma(a)d-\sigma(c)b=1\right\}\subseteq \GL_2(A)$$

$$\spp_2(A,\sigma)=\left\{\begin{pmatrix}
    x & z \\
    y & -\sigma(x)
    \end{pmatrix}\midwd x\in A,\;y,z\in A^\sigma\right\}\subseteq \Mat_2(A).$$

We now assume $(A,\sigma)$ to be Hermitian.

\begin{prop}[{\cite[Lemman~2.7.3]{R-Thesis}}]
For a Hermitian algebra $(A,\sigma)$ the group $\Sp_2(A,\sigma)$ is a connected Lie group.
\end{prop}

We describe a maximal compact subgroup of $\Sp_2(A,\sigma)$.

\begin{df}
We denote:
$$\UU_2(A,\sigma):=\{M\in\Mat_2(A)\mid \sigma(M)^TM=\Id\};$$
$$\KSp_2(A,\sigma):=\Sp_2(A,\sigma)\cap\UU_2(A,\sigma).$$
\end{df}

\begin{teo}[{\cite[Theorem~3.19]{ABRRW}}]\label{maxcomp-Sp_R} The group
$\KSp_2(A,\sigma)$ is a maximal compact subgroup of $\Sp_2(A,\sigma)$.
\end{teo}

Let $A_\CC$ be the complexification of the algebra $A$, i.e. $A_\CC:=A\otimes_\R\CC=\{a_1+ia_2\mid a_1,a_2\in A\}$.

\begin{teo}[{\cite[Proposition~5.8]{ABRRW}}]\label{model_of_ss}
The group $\Sp_2(A,\sigma)$ acts on
$$A^\sigma+i A^\sigma_+=\{a_1+ia_2\mid a_1\in A^\sigma,\;a_2\in A^\sigma_+\}\subset A_\CC$$
transitively by generalized M\"obius transformations: for $z\in A^\sigma+i A^\sigma_+$ and $g=\begin{pmatrix} c_{11} & c_{12}\\ c_{21} & c_{22}\end{pmatrix}$
$$g(z):=(c_{11}z+c_{12})(c_{21}z+c_{22})^{-1}.$$
The stabilizer of the point $i\in A^\sigma+i A^\sigma_+$ is $\KSp_2(A,\sigma)$. In particular, $A^\sigma+i A^\sigma_+$ is a model of the symmetric space of $\Sp_2(A,\sigma)$.
\end{teo}

Theorem~\ref{model_of_ss} shows in particular that the symmetric space of $\Sp_2(A,\sigma)$ is a non-compact complex manifold.

\subsection{Hermitian symmetric spaces and Hermitian Lie groups}\label{sec:herm_lie}

Let $M$ be a connected complex manifold, $g$ be a Hermitian metric on $M$.
\begin{df}
$(M,g)$ is called \emph{Hermitian symmetric space} if every $p\in M$ is an isolated fixed point of a holomorphic involutive isometry $s_p\colon M\to M$.
\end{df}

\begin{teo}[{Harish-Chandra embedding theorem~\cite[Theorem~7.1]{Helgason}}]
Every non-compact Hermitian symmetric space is biholomorphic to an open bounded domain in a complex vector space.
\end{teo}

\begin{df}
A Hermitian symmetric space is called \emph{of tube type} if it is biholomorphic to a domain of the form $V + i\Omega$ where $V$ is a real vector space
and $\Omega\subset V$ is an open proper convex cone.
\end{df}

Let $G$ be a connected non-compact semisimple Lie group with finite (or more generally with compact) center, $K$ be a maximal compact subgroup of $G$.
\begin{df}
The group $G$ is called \emph{Hermitian Lie group} if the symmetric space $G/K$ is Hermitian. $G$ is called \emph{Hermitian Lie group of tube type} if the symmetric space $G/K$ is Hermitian of tube type.
\end{df}

\begin{cor}
Let $(A,\sigma)$ be a Hermitian algebra. From Theorem~\ref{model_of_ss} follows that $\Sp_2(A,\sigma)$ is a Hermitian Lie group of tube type.
\end{cor}

\begin{fact}
Hermitian Lie groups of tube type are classified. Up to coverings, these are:
$$\Sp(2n,\R),\;\;\SU(n,n),\;\; \SO^*(4n),\;\; \SO_0(2,n+2),\;\; E_{7(-25)}$$
for $n\in \N$. The first four groups are called \defin{classical}, the last one is \defin{exceptional}.
\end{fact}

\subsection{Realization of classical Hermitian Lie groups of tube type as \texorpdfstring{$\Sp_2(A,\sigma)$}{Sp2(A,sigma)}}\label{ex:groups}

If $(A,\sigma)$ is a Hermitian algebra, then the symplectic group $\Sp_2(A,\sigma)$ is isomorphic to classical Hermitian Lie groups of tube type.

\begin{enumerate}
    \item Real symplectic group $\Sp(2n,\R)$. To realize the real symplectic group $\Sp(2n,\R)$, we take $A=\Mat(n,\R)$ to the be algebra of $n\times n$ matrices over $\R$ and consider the involution $\sigma: A \to A$ given by $\sigma(r)=r^T$, $r\in A$. Then $\Sp_2(A,\sigma)$ is isomorphic to $\Sp(2n,\R)$. Its maximal compact subgroup $\KSp_2(A,\sigma)$ is isomorphic to $\UU(n)$.

    \item Indefinite unitary group $\UU(n,n)$. To realize the indefinite unitary group $\UU(n,n)$, we consider $A=\Mat(n,\CC)$, and the involution $\bar\sigma: A \to A$ given by $\bar\sigma(r)=\bar r^T$. Then $\Sp_2(A,\bar\sigma)$ is isomorphic to $\UU(n,n)$. To see this, we notice that the standard Hermitian form $h$ of signature $(n,n)$ on $\CC^{2n}$ is given by $h(x,y):=i\omega(xT,yT)$ where $T=\diag(\Id_n,-i\Id_n)$. Further, $\KSp_2(A,\sigma)$ is isomorphic to $\UU(n)\times\UU(n)$. Notice that we cannot realize the special unitary group $\mathrm{SU}(n,n)$ as $\Sp_2(A,\sigma)$. But since the center of $\UU(n,n)$ is compact, it does not affect the symmetric space.

    \item The group $\SO^*(4n)$. By definition, the group $\SO^*(2n)$ (some authors also use the notation $\OO(n,\HH)$) is the group of isometries of the following form on the quaternionic right module $\HH^n$:
        $$\beta(x,y)=\sum_{i=1}^{n} \bar x_i j y_i$$
        where $x=(x_1,\dots,x_n),y=(y_1,\dots y_n)\in\HH^n$ and $\bar\cdot\colon\HH\to\HH$ is the quaternionic conjugation.

        The groups $\SO^*(2n)$ are Hermitian Lie groups, but they are of tube type only if $n$ is even. In order to realize $\SO^*(4n)$ as $\Sp_2(A,\sigma)$ we consider  $A=\Mat(n,\HH)$ and the involution $\sigma_1: A \to A$, given by $\sigma_1(r)=\bar r^T=\bar\sigma(r_1)-\sigma(r_2)j$ for $r=r_1+r_2j$ and $r_1,r_2\in \Mat(n,\CC)$. Then $\Sp_2(A,\sigma_1)$ is isomorphic $\SO^*(4n)$ considered as the group of isometries of the quaternionic form $\beta$ on $\HH^{2n}$ defined by
        $$\beta(x,y)=\sum_{i=1}^{2n} \bar x_i j y_i=\bar x^T(\Id_{2n}j)y.$$
        To see this, we notice that
        $$\Id_{2n} j=\sigma_1(T)\Ome{\Id_n} T$$
        for
        $$T=\frac{1}{\sqrt{2}}\begin{pmatrix}\Id_n & -\Id_n j \\-\Id_n j & \Id_n\end{pmatrix}.$$
        The group $\KSp_2(A,\sigma)$ is isomorphic to $\UU(2n)$.
\end{enumerate}

\begin{rem}
The Hermitian Lie group of tube type $\SO_0(2,n)$ cannot be realized in the same way as $\Sp_2(A,\sigma)$. Nevertheless, the double cover $\Spin_0(2,n)$ of $\SO_0(2,n)$ can be realized as $\Sp_2$ over some more complicated object which we do not discuss in this paper (for more details about this realization see~\cite[Section~2.10]{R-Thesis}).
\end{rem}

\begin{rem}\label{rem:exceptional}
The exceptional Hermitian Lie group of tube type $E_{7(-25)}$ cannot be realized as $\Sp_2(A,\sigma)$. The reason is that the Euclidean Jordan algebra $\Herm(3,\Oc)$ cannot be embedded into any associative algebra, in particular it cannot be realized as $A^\sigma$ for a Hermitian algebra $(A,\sigma)$. For more details about Euclidean Jordan algebras and their embeddings into associative algebras we refer to~\cite{Faraut,Hanche84}.
\end{rem}

\section{Space of isotropic lines}\label{sec:is_islines}

We assume $(A,\sigma)$ to be Hermitian. We denote by $\PP(A^2)$ the space of lines in $A^2$, i.e.
$$\PP(A^2):=\{xA\mid x\in A^2\text{ regular}\}.$$
The group $\GL_2(A)$ acts on $\PP(A^2)$.

If $A^2$ is equipped with the standard symplectic form $\omega$, we denote by $\PP(\Is(\omega))$ the space of all isotropic (with respect to $\omega$) lines:
$$\PP(\Is(\omega))=\{xA\subset A^2\mid \omega(x,x)=0,\,x\text{ regular}\}.$$
This space is a closed subspace of $\PP(A^2)$. The group $\Sp_2(A,\sigma)$ acts on the space of isotropic lines.

\subsection{Spaces of lines as compact symmetric spaces}
The spaces $\PP(A^2)$ and $\PP(\Is(\omega))$ can be seen as compact symmetric spaces as follows:

\begin{prop}[{\cite[Proposition~4.1]{ABRRW}}]\label{PP_compact}
Let $(A,\sigma)$ be Hermitian. The group $\UU_2(A,\sigma)$ acts transitively on $\PP(A^2)$ with
$$\Stab_{\UU_2(A,\sigma)}\Clm{1}{0}A=\left\{\begin{pmatrix}u_1 & 0 \\ 0 & u_2\end{pmatrix}\midwd u_1,u_2\in U_{(A,\sigma)}\right\}\cong U_{(A,\sigma)}\times U_{(A,\sigma)}.$$
In particular, the space $\PP(A^2)$ is compact and homeomorphic to the quotient space:
$$\UU_2(A,\sigma)/U_{(A,\sigma)}\times U_{(A,\sigma)}$$
where the group $U_{(A,\sigma)}\times U_{(A,\sigma)}$ is embedded into $\UU_2(A,\sigma)$ diagonally.
\end{prop}

\begin{prop}[{\cite[Corollary~4.3]{ABRRW}}]
Let $(A,\sigma)$ be Hermitian and $\omega$ be the standard symplectic form on $A^2$. Then $\KSp_2(A,\sigma)$ acts transitively on $\PP(\Is(\omega))$ with
$$\Stab_{\KSp_2(A,\sigma)}\Clm{1}{0}A=\left\{\begin{pmatrix}u & 0 \\ 0 & u\end{pmatrix}\midwd u \in U_{(A,\sigma)}\right\}\cong U_{(A,\sigma)}.$$
In particular, $\PP(\Is(\omega))$ is homeomorphic to the quotient space:
$$\KSp_2(A,\sigma)/U_{(A,\sigma)}$$
where the group $U_{(A,\sigma)}$ is embedded into $\KSp_2(A,\sigma)$ in the diagonal way.
\end{prop}

\subsection{Action of \texorpdfstring{$\Sp_2(A,\sigma)$}{Sp2(A,sigma)} on tuples of isotropic lines}
In this section we understand the action of $\Sp_2(A,\sigma)$ on the space of isotropic lines, on transverse pairs of isotropic lines, maximal triples and positive quadruples of isotropic lines and describe invariants of this action. The results stated here are proven in~\cite[Chapter~4]{ABRRW}.

To simplify the notation, we denote $G:=\Sp_2(A,\sigma)$ and $\F:=\PP(\Is(\omega))$. Further, $\ell^+:=(1,0)^TA$, $\ell^-:=(0,1)^TA$, $\ell^1:=(1,1)^TA$. These are three pairwise transverse isotropic lines of $\F$. We denote:
$$\Pp\coloneqq\Stab_G(\ell^+),\;\Pm\coloneqq\Stab_G(\ell^-),$$
$$\L\coloneqq\Pp\cap\Pm=\Stab_G(\ell^+,\ell^-),$$
$$\K\coloneqq\Stab_G(\ell^+,\ell^-,\ell^1).$$

\begin{prop}[Action on isotropic lines]\label{stab1_A} The group $G$ acts transitively on $\F$ and
$$\Pp=\left\{
\begin{pmatrix}
x & xy \\
0 & \sigma(x)^{-1}
\end{pmatrix}
\midwd
x\in A^\times, y\in A^\sigma
\right\},\;\text{and }\;
\Pm=\left\{
\begin{pmatrix}
x & 0 \\
zx & \sigma(x)^{-1}
\end{pmatrix}
\midwd
x\in A^\times, z\in A^\sigma
\right\}.$$
In particular, $\F$ is homeomorphic to the homogeneous space $G/\Pp$.
\end{prop}

\begin{prop}[Action on transverse pairs of isotropic lines]\label{stab2_A}
The group $\Sp_2(A,\sigma)$ acts transitively on pairs of transverse isotropic lines and
$$\L=\left\{
\begin{pmatrix}
x & 0 \\
0 & \sigma(x)^{-1}
\end{pmatrix}
\midwd
x\in A^\times\right\}\cong A^\times.$$
\end{prop}

Let $(l_1,l_2,l_3)$ be a triple of pairwise transverse isotropic lines. Because $l_1$ and $l_2$ are transverse, there exist $g\in\Sp_2(A,\sigma)$ such that $g(l_1,l_2)=(\ell^+,\ell^-)$. Because $l_3$ and $l_1$ are transverse, there exist an element $b\in A$ such that $l_3=(b,1)^TA$. Because $l_3$ is isotropic, $b\in A^\sigma$, and because $l_3$ and $l_2$ are transverse, $b\in (A^\sigma)^\times$. The element $b$ is determined by the triple $(l_1,l_2,l_3)$ up to congruence by an element of $A^\times$.

Since the action by congruence preserves $A^\sigma_+$, we can give the following definition:

\begin{df}
A triple $(l_1,l_2,l_3)$ is called \defin{maximal} if the element $b$ defined as above is in $A^\sigma_+$.
\end{df}

\begin{rem}
The notion ``maximal'' is used because if $(A,\sigma)$ is a Hermitian algebra, then for every element of $A^\sigma$ a signature can be assigned which is a bounded integer number. This signature is maximal for elements of $A^\sigma_+$. For more details about the signature we refer to~\cite[Section~2.3]{ABRRW}.
\end{rem}

\begin{prop}[Action on maximal triples of isotropic lines]\label{act_pos_triples}
The group $\Sp_2(A,\sigma)$ acts transitively on maximal triples of isotropic lines and
$$\K=\left\{
\begin{pmatrix}
u & 0 \\
0 & u
\end{pmatrix}
\midwd u\in U_{(A,\sigma)}\right\}\cong U_{(A,\sigma)}.$$
\end{prop}

\begin{df}
A quadruple $(l_1,l_2,l_3,l_4)$ of isotropic lines is called \defin{positive} if the triples $(l_1,l_2,l_3)$ and $(l_1,l_3,l_4)$ are maximal.
\end{df}

\begin{prop}[Action on positive quadruples of isotropic lines]\label{act_pos_quadr}
Let $(l_1,l_2,l_3,l_4)$ be a positive quadruple. There exists an element $g\in\Sp_2(A,\sigma)$ and $b\in A^\sigma_+$ such that
$g(l_1,l_2,l_3,l_4)=(\ell^+,\ell(b),\ell^-,\ell^1)$ where $\ell(b)=(1,-b)^TA$.
The element $g$ is unique up to the left multiplication by an element of $\K$. The element $b$ is determined by the quadruple $(l_1,l_2,l_3,l_4)$ uniquely up to conjugation by an element of $U_{(A,\sigma)}$.
\end{prop}

\section{Topological data}\label{sec:top_data}

\subsection{Punctured surfaces}
Let $\bar S$ be a compact oriented smooth surface of finite type with or without boundary. Let $P$ be a nonempty finite subset of points of $\bar S$ such that on every boundary component there is at least on element of $P$. We define $S\coloneqq\bar S\bs P$. We assume that the Euler characteristic $\chi(S)$ of $S$ is negative or that $\bar S$ is diffeomorphic to a closed disc and there are at least three punctures. Surfaces that can be obtained in this way are called \defin{punctured surfaces} (some authors call them also ciliated surfaces, e.g.~\cite{palesi}). Elements of $P$ are called \defin{punctures} of $S$. Sometimes we will distinguish between elements of $P$ that lie in the interior of $\bar S$ -- \defin{internal punctures} and that lie on the boundary -- \defin{external punctures}.

Every punctured surface can be equipped with a complete hyperbolic structure of finite volume with geodesic boundary. For every such hyperbolic structure all the punctures are cusps and all boundary curves are (infinite) geodesics. Once equipped with a hyperbolic structure as above, the universal covering $\tilde S$ of $S$ can be seen as closed subset of the hyperbolic plane $\HH^2$ which is invariant under the natural action of $\pi_1(S)$ on $\HH^2$ by the holonomy representation. Punctures $P$ of $S$ are lifted to points of the ideal boundary of $\HH^2$ which we call punctures of $\tilde S$ and denote by $\tilde P\subset \partial_\infty \tilde S\subseteq \partial_\infty\HH^2$. Notice that if $\bar S$ has no boundary, then $\tilde S$ is the entire $\HH^2$.

\subsection{Ideal triangulations}
An \defin{ideal triangulation} of $S$ is a triangulation of $\bar S$ whose set of vertices agrees with $P$. We always consider edges of an ideal triangulation as homotopy classes of (non-oriented) paths connecting points in $P$ (relative to its endpoints). Connected components of the complement on $S$ to all edges of an ideal triangulation $\mathcal T$ are called faces or triangles of $\mathcal T$. Every edge belongs to the boundary of one or two triangles. In the first case, an edge is called \defin{external}, in the second -- \defin{internal}. We denote by $E$, resp. $E_{in}$, resp. $E_{ex}$ the set of all edges, resp. all internal edges, resp. all external edges of $\T$. We denote by $T$ the set of faces of $\T$ and fix a total order $<$ on $T$. Any ideal triangulation of $S$ can be realized as an ideal geodesic triangulation as soon as a hyperbolic structure as above on $S$ is chosen. We fix a hyperbolic structure of finite volume on $S$ and assume that edges of $\T$ are geodesics.

The following lemma counts the number of triangles of an ideal triangulation of $S$:

\begin{lem}
Let $S$ be a surface of genus $g$ with $p_i$ internal punctures and $p_e$ external punctures. Let $m$ be the number of boundary components of $\bar S$, then
$$\#T=4g-4+2p_i+2m+p_e=p_e-2\chi(S).$$
where $\chi(S)=2-2g-p_i-m$ is the Euler characteristic of $S$.
\end{lem}

\subsection{Graph \texorpdfstring{$\Gamma$}{Gamma}}\label{sec:Graph_Gamma}

We construct the following graph $\Gamma$ on $S$. The set of vertices of $\Gamma$ is $V_\Gamma\coloneqq\{(\tau,r)\in T\times E\mid r\subset\bar \tau\}$. A vertex $(\tau,r)\in V_\Gamma$ can be seen as a point in the triangle $\tau\in T$ lying close to the internal edge $r\in E$.

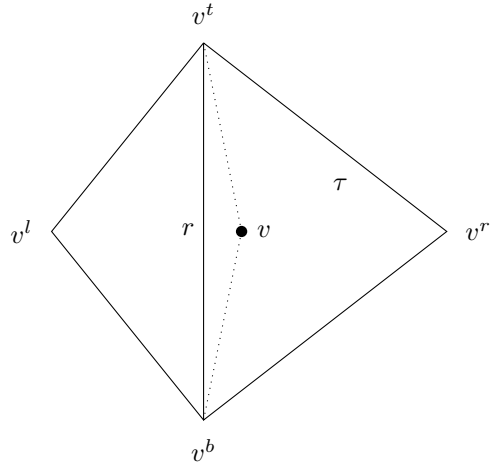
\begin{figure}[ht]
\begin{tikzpicture}

\draw (-3,3) node[label = $v^t$] (v1) {} -- (-3,-2) node[label = below:$v^b$] (v2) {} -- (0.2,0.5) node[label= right:$v^r$] (v3) {} -- (-3,3);
\draw (-3,3) node[ ] (v1) {} -- (-5,0.5) node[label = left:$v^l$] (v2) {} -- (-3,-2) node[ ] (v3) {};

\draw [dotted](-3,3) -- (-2.5,0.5);
\draw [dotted](-3,-2) -- (-2.5,0.5);
\node[fill, circle, inner sep = 1.5pt, label= right:$v$] at (-2.5,0.5) {};
\node[label= $r$] at (-3.2,0.2) {};
\node[label= $\tau$] at (-1.2,0.8) {};

\end{tikzpicture}
\caption{Definition of $v^t$, $v^b$, $v^r$ and $v^l$ for $v\in V_\Gamma$}\label{vt-vb}
\end{figure}

Let $v=(\tau,r)\in V_\Gamma$ and the edge $r$ connects two punctures $p,p'\in P$. If we connect $v$ with $p$ and $p'$ by two simple non intersecting segments in $\tau$, we obtain a triangle with vertices $v,p,p'$ in $\tau$. We denote $v^t\coloneqq p$, $v^b\coloneqq p'$ if the orientation of the triangle $(v,p,p')$ agrees with the orientation of the surface (see Figure~\ref{vt-vb}). Further we define $v^r$ as the unique puncture of $P$ such that $(v^t,v^b,v^r)=\tau$. If $r$ is an internal edge of $\mathcal T$, we also define $v^l$ as the unique puncture of $P$ such that $(v^b,v^t,v^l)$ is a triangle of $\mathcal T$. This triangle is adjacent to $\tau$ along $r$.

\begin{figure}[ht]
\begin{tikzpicture}[every node/.style={fill, circle, inner sep = 1.5pt}]
\draw (-5,0) -- (1,4.5) -- (7,0.5) -- (1,-3.5) -- (-5,0);
\draw (1,4.5) -- (1,-3.5);

\node[white, label=$\tau$] at (2,3) {};
\node[white, label=$\tau'$] at (0.2,3) {};

\node[label=above left:$v_4$](v4) at (1.8,0.5) {};
\node[label=above right:$v_3$](v3) at (-2+0.3,2-0.3) {};
\node[label=above right:$v_5$](v5) at  (3.5-0.2,-1.5+0.2) {};
\node[label=below:$v_2$](v2) at (0.2,0.5) {};
\node[label=right:$v_1$](v1) at (-2,-1) {};
\node[label=right:$v_6$](v10) at (3.5,2) {};

\draw[middlearrow={latex}] (v1) to (v2) ;


\draw[middlearrow={latex}] (v2) to (v3) ;

\draw[middlearrow={latex}] (v3) to (v1) ;

\draw[middlearrow={latex}] (v4) to (v5) ;
\draw[middlearrow={latex}] (v5) to (v10) ;
\draw[middlearrow={latex}] (v10) to (v4) ;

\node (v6) at (-2.5,2.5) {};
\node (v7) at (4,-2) {};
\node (v8) at (-2.2,-2.2) {};

\draw[middlearrow={latex}] (v4) to (v2) ;

\draw[middlearrow={latex}] (v3) to (v6) ;


\draw[middlearrow={latex}] (v5) to (v7) ;

\draw[middlearrow={latex}] (v8) to (v1) ;

\end{tikzpicture}
\caption{Example of a graph $\Gamma$. Here the upper right edge of the triangulation is external, all other edges are internal and $\tau<\tau'$.}\label{Graph_Gamma}
\end{figure}

We now describe the set of oriented edges $E_\Gamma^+$ of $\Gamma$. First we fix the following notation: the oriented edge from the vertex $v$ to the vertex $v'$ is denoted by $(v\to v')$ and $(v'\gets v)$.
\begin{itemize}
    \item Let $v=(\tau_1,r)$, $v'=(\tau_2,r)$ and $\tau_1<\tau_2$, then $(v \to v')$ is an edge of $\Gamma$. It crosses the internal edge $r$ of the triangulation $\mathcal T$.
    \item Let $v_1$ and $v_2$ be two vertices lying in one triangle such that $v_1^b=v_2^t$, then $(v_1\to v_2)$ is an edge of $\Gamma$ (see Figure~\ref{Graph_Gamma}).
\end{itemize}

\begin{rem}
The graph $\Gamma$ has no multiple edges and no two-cycles.
\end{rem}

We also denote $E^-_\Gamma\coloneqq \{(v'\gets v)\mid (v\gets v')\in E^+_\Gamma\}$. Since $\Gamma$ has no two-cycles, $E^+_\Gamma\cap E^-_\Gamma=\emptyset$. We denote $E_\Gamma\coloneqq E^+_\Gamma\cup E^-_\Gamma$.

We denote by $\tilde\Gamma$ the lift of $\Gamma$ to the universal covering $\tilde S\subset \HH^2$ of $S$. We also choose a point $b\in V_\Gamma$ and fix one of its lifts $\tilde b\in V_{\tilde\Gamma}$. Assume $b\in \tau_0$ and $\tilde b\in \tilde \tau_0$ for $\tau_0\in T$ and $\tilde \tau_0\in \tilde T$.

\begin{rem}
Every path on $S$ that starts and ends in $V_\Gamma$ can be deformed to a non-oriented path in $\Gamma$.
\end{rem}

\section{Framed representations and framed local systems}\label{sec:framed_reps}
\subsection{Transverse and maximal framed representations}
Let $S$ be a punctured surface, and let $(A,\sigma)$ be a Hermitian algebra. As before, we denote $G:=\Sp_2(A,\sigma)$ and $\F:=\PIs$.

\begin{df}
The space of all homomorphisms $\rho\colon\pi_1(S)\to G$ is denoted by $\Hom(\pi_1(S),G)$. The group $G$ acts on $\Hom(\pi_1(S),G)$ by conjugation. The quotient space is denoted by $\Rep(\pi_1(S),G):=\Hom(\pi_1(S),G)/G$.
\end{df}

\begin{df}
A \defin{framing} is a map $F\colon \tilde P\to \F$. Let $\rho\colon\pi_1(S)\to G$ be a homomorphism. A \defin{framing of $\rho$} is a $\pi_1(S)$-equivariant framing $F\colon \tilde P\to \F$, i.e. for every $\gamma\in\pi_1(S)$, $F(\gamma(\tilde p))=\rho(\gamma)F(\tilde p)$ for all $\tilde p\in \tilde P$. A \defin{framed homomorphism} is a pair $(\rho,F)$ where $F$ is a framing of $\rho$.

The space of all framed homomorphisms is denoted by $\Hom^{\fr}(\pi_1(S),G)$. The space
$$\Rep^{\fr}(\pi_1(S),G)\coloneqq \Hom^{\fr}(\pi_1(S),G)/G$$
is called the \defin{space of framed representations}. A \defin{framed representation} is an element of $\Rep^{\fr}(\pi_1(S),G)$.
\end{df}

\begin{rem}
Notice that not every homomorphism $\rho\colon\pi_1(S)\to G$ admit a framing. Homomorphisms that admit a framing are usually called \defin{peripherally parabolic}.
\end{rem}

Let $\T$ be an ideal triangulation of $S$ and $\tilde\T$ is the lift of $\T$ to the universal covering $\tilde S$.

\begin{df}\label{df:transverse_rep}
A framed homomorphism $(\rho,F)$ is called \defin{$\T$-transverse} if for every two punctures $\tilde p_1,\tilde p_2\in\tilde P$ that are connected by an edge of $\tilde\T$, the isotropic lines $F(\tilde p_1)$ and $F(\tilde p_2)$ are transverse.

The space of all $\T$-transverse framed homomorphisms is denoted by $\Hom^{\fr}_\T(\pi_1(S),G)$. The space
$$\Rep^{\fr}_\T(\pi_1(S),G)\coloneqq \Hom^{\fr}_\T(\pi_1(S),G)/G$$
is called the \defin{space of $\T$-transverse framed representations}.
\end{df}

\begin{df} A framing $F$ is called \defin{maximal} if for every cyclically oriented triple $(\tilde p_1,\tilde p_2,\tilde p_3)\in\tilde P^3$, the triple of isotropic lines $(F(\tilde p_1),F(\tilde p_2),F(\tilde p_3))$ is maximal.
A homomorphism $\rho\in\Hom(\pi_1(S),G)$ is called \defin{maximal} if it admits a maximal framing $F$. The pair $(\rho, F)$ is called a \defin{maximal framed homomorphism}.

The space of all maximal framed homomorphisms is denoted by $\Hom^{\fr}_+(\pi_1(S),G)$. The space
$$\Rep^{\fr}_+(\pi_1(S),G)\coloneqq \Hom^{\fr}_+(\pi_1(S),G)/G$$
is called the \defin{space of maximal framed representations}.
\end{df}

\begin{rem}
From this definition follows immediately that maximal representations are transverse with respect to any ideal triangulation of $S$.
\end{rem}

\begin{rem}\label{max.Hausdorff}
Notice, that in general not all orbits of the $G$-actions on $\Hom(\pi_1(S),G)$ and on $\Hom^{\fr}(\pi_1(S),G)$ are closed. So the quotient spaces  $\Rep(\pi_1(S),G)$ and $\Rep^{\fr}(\pi_1(S),G)$ are not Hausdorff. However the space of maximal representations (that will be defined in the next section) inside $\Rep(\pi_1(S),G)$ and the space of maximal framed representations $\Rep^{\fr}_+(\pi_1(S),G)$ inside $\Rep^{\fr}(\pi_1(S),G)$ form Hasudorff subspaces.
\end{rem}

In fact, to check if the framing is maximal, it is enough to check it only for all positively oriented triangles of one triangulation and not for all possible cyclically oriented triples in $\tilde P$.

\begin{prop}
Let $\mathcal T$ be a triangulation of $S$ and let $F$ be a framing such that for every positively oriented triangle $(\tilde p_1,\tilde p_2,\tilde p_3)$ of $\tilde{\mathcal T}$  the triple of isotropic lines $(F(\tilde p_1),F(\tilde p_2),F(\tilde p_3))$ is maximal. Then the framing $F$ is maximal.
\end{prop}

This statement is proven  in~\cite[Section~4.12]{AGRW} for $G=\Sp(2n,\R)$. The proof in general case is similar.

\subsection{Maximal representations into Hermitian Lie groups of tube type}

In this section we assume that $S$ has no boundary, i.e. that all punctures are internal. We also assume that $G$ is a Hermitian Lie group of tube type (not necessarily $\Sp_2(A,\sigma)$).

An important invariant of a homomorphism $\rho\colon \pi_1(S)\to G$ is the Toledo invariant $T_\rho$, which was introduced in~\cite{BIW} using bounded cohomology. It is important to emphasize that the Toledo number depends on the topological surface $S$ and not only on its fundamental group. It is a real number which satisﬁes the Milnor–Wood inequality:
$$-\rk_\R(G)|\chi(S)|\leq T_\rho\leq  \rk_\R(G)|\chi(S)|$$
where $\rk_\R$ is the real rank of the group $G$. Moreover, $T_\rho$ is invariant under the action of $G$ on $\Hom(\pi_1(S),G)$ by conjugation. So it is well defined for representations in $\Rep(\pi_1(S),G)$.

\begin{df}
A representation $\rho\in\Rep(\pi_1(S),G)$ is called \defin{maximal} if $T_\rho=\rk_\R(G)|\chi(S)|$. The subspace of all maximal representations in $\Rep(\pi_1(S),G)$ is denoted by $\Rep_+(\pi_1(S),G)$.
\end{df}

Maximal representations have nice geometric properties that we state in the next proposition. They were proved in~\cite{BIW} and in~\cite{AGRW} for symplectic representations. The general proof for Hermitian Lie groups of tube type is similar.

\begin{prop}\label{max.prop} Let $S$ be a punctured surface without boundary and $G$ be a Hermitian Lie group of tube type.
\begin{itemize}
    \item Every maximal homomorphism $\rho$ is discrete and faithful, i.e. $\rho$ has no kernel and the image of $\rho$ is discrete in $G$.
    \item Maximal homomorphisms are reductive; hence $\Rep_+(\pi_1(S),G)$ is Hausdorff (see Remark~\ref{max.Hausdorff}).
    \item Let $G=\Sp_2(A,\sigma)$ where $(A,\sigma)$ is a Hermitian algebra. Every maximal homomorphism $\rho\colon \pi_1(S)\to G$ admit a framing. Every framing of $\rho$ is transverse with respect to any ideal triangulation of $S$. Let $F$ be a framing of $\rho$, then $(\rho,F)$ is a maximal framed representation.
\end{itemize}
\end{prop}

\subsection{Framed trivial local systems on graphs}
Let $\Gamma$ be an oriented graph without multiple edges and two-cycles. We denote by $V_\Gamma$ the set of vertices of $\Gamma$ and by $E^+_\Gamma$ the set of oriented edges. If $e\in E^+_\Gamma$ is an edge going from the vertex $v\in V_\Gamma$ to the vertex $v'\in V_\Gamma$, we write $e=(v\to v')=(v'\gets v)$. We also denote $E^-_\Gamma\coloneqq \{(v'\gets v)\mid (v\gets v')\in E^+_\Gamma\}$. Since $\Gamma$ has no two-cycles, $E^+_\Gamma\cap E^-_\Gamma=\emptyset$. We denote $E_\Gamma\coloneqq E^+_\Gamma\cup E^-_\Gamma$.

\begin{df}
A \defin{transverse framing} of $\Gamma$ is a pair of maps $F^t,F^b\colon V_\Gamma\to \F$ such that $F^t(v)$ and $F^b(v)$ are transverse for every $v\in V_\Gamma$.
\end{df}

\begin{df}\label{df:triv_loc_sys} A \defin{trivial $G$-local system on $\Gamma$} is a map $T\colon V_\Gamma\cup E_\Gamma\to G$ such that for any edge $(v'\gets v)\in E_\Gamma$, $v,v'\in V_\Gamma$, $T(v'\gets v)=T(v\gets v')^{-1}$ and  $T(v')=T(v'\gets v)T(v)$. A trivial $G$-local system on $\Gamma$ defined by the map $T$ is denoted by $(\Gamma,T)$. 
\end{df}

\begin{rem}
From Definition~\ref{df:triv_loc_sys} follows that $T(e_k)...T(e_1)=1$ for every cycle $(e_1,\dots e_k)$ where all $e_i\in E_\Gamma$, i.e. holonomies around cycles are trivial. This is why we call such local system trivial.

Let $v\in V_\Gamma$ be a fixed vertex of $\Gamma$. The map $T$ is determined by $T|_{E^+_\Gamma}$ and $T(v)$. The elements $T(v)$ and $T(e)$ for all $e\in E^+_{\Gamma}$ can be arbitrary elements of $G$ satisfying the following condition: $T(e_k)...T(e_1)=1$ for every cycle $(e_1,\dots e_k)$ where all $e_i\in E_\Gamma$.
\end{rem}

Let $(F^t,F^b)$ be a transverse framing, let $(\Gamma,T)$ be a trivial $G$-local system on $\Gamma$. We say that the transverse framing $(F^t,F^b)$ is \defin{adapted} to the trivial local system $(\Gamma,T)$ if for every $v\in V_\Gamma$ holds: $T(v)F^t(v)=\ell^+$ and $T(v)F^b(v)=\ell^-$. A quadruple $\mathfrak L:=(\Gamma,T,F^t,F^b)$ is called a \defin{transverse framed trivial $G$-local system} if $(F^t,F^b)$ is a transverse framing adapted to $(\Gamma,T)$. 

Since $\Gamma$ has no multiple edges, any path $\gamma$ in $\Gamma$ can be written in a unique way as a sequence of vertices $\gamma=(v_k\gets v_{k-1}\gets\dots\gets v_1)$ where $v_1$ is the starting vertex of $\gamma$ and $v_k$ is the end-vertex of $\gamma$ and $(v_{i+1}\gets v_{i})\in E_\Gamma$ for all $i\in\{1,\dots,k-1\}$. 
We also extend the map $T$ to the space of paths as follows: if $\gamma=(v_k\gets v_{k-1}\gets\dots\gets v_1)$, then $T(\gamma)\coloneqq T(v_k\gets v_{k-1})\dots T(v_2\gets v_1)$.

\begin{rem}
Let $\gamma$ be a path in $\Gamma$ connecting two vertices $v$ and $w$. Then in fact, $T(\gamma)$ depend only on $v$ and $w$ and not on the entire path $\gamma$. So sometimes to emphasise this fact, we will just write $T(w\gets v)$ instead of $T(\gamma)$.
\end{rem}

\subsection{From local systems to representations}\label{Loc_to_Rep}

Let $\mathfrak L=(\Gamma,T,F^t,F^b)$ be a transverse framed trivial $G$-local system on $\Gamma$. Let $H$ be a subgroup of the group $\Aut(\Gamma)$ of all automorphisms of $\Gamma$.

\begin{df}
The local system $\mathfrak L=(\Gamma,T,F^t,F^b)$ is called \defin{$H$-invariant} if for all $\gamma\in H$ and for all edges $(w\gets v)\in E_\Gamma$, $v,w\in V_\Gamma$, $T(w\gets v)=T(\gamma w\gets \gamma v)$.
\end{df}

An $H$-invariant $G$-local system always gives rise to a homomorphism $\rho\colon H\to G$ in the following way: let $v\in V_\Gamma$, we define for every $\gamma \in H$:
\begin{equation}\label{Graph_representation}
\rho_\mathfrak L(\gamma)\coloneqq T(\gamma v)^{-1}T(v).
\end{equation}

\begin{rem}
If $\mathfrak L=(\Gamma,T,F^t,F^b)$ is an $H$-invariant local system, then the framing $(F^t,F^b)$ is $\rho_\mathfrak L$-equivariant, i.e. for all $v\in V_\Gamma$, $F_t(\gamma v)=\rho_\mathfrak L(\gamma)(F_t(v))$ and $F_b(\gamma v)=\rho_\mathfrak L(\gamma)(F_b(v))$.
\end{rem}

\begin{prop}
The map $\rho_\mathfrak L$ is a group homomorphism that does not depend on the choice of $v\in V_\Gamma$.
\end{prop}

\begin{proof}
Let $w\in V_\Gamma$ be another vertex. Then
$$T(\gamma w)^{-1}T(w)
=(T(\gamma w\gets \gamma v)T(\gamma v))^{-1}(T(w\gets v)T(v))$$
$$=T(\gamma v)^{-1}T(\gamma w\gets \gamma v)^{-1}T(w\gets v)T(v)
=T(\gamma v)^{-1}T(v)$$
since $T(\gamma w\gets \gamma v)=T(w\gets v)$ by $H$-invariance. So $\rho_\mathfrak L$ does not depend on the choice of $v\in V_\Gamma$.

Let now $\gamma_1,\gamma_2\in H$. We obtain $\rho_\mathfrak L(\gamma_2)=T(\gamma_2 (\gamma_1 v))^{-1}T(\gamma_1 v)$ and $\rho_\mathfrak L(\gamma_1)=T(\gamma_1 v)^{-1}T(v)$. Therefore,
$$\rho_\mathfrak L(\gamma_2)\rho_\mathfrak L(\gamma_1)=T(\gamma_2 (\gamma_1 v))^{-1}T(\gamma_1 v)T(\gamma_1 v)^{-1}T(v)=T(\gamma_2 (\gamma_1 v))^{-1}T(v)=\rho_\mathfrak L(\gamma_2\gamma_1).$$
So $\rho_\mathfrak L$ is a group homomorphism.
\end{proof}

Let $S$ be a punctured surface as before. Let $\mathcal T$ be an ideal triangulation of $S$ and $\tilde\Gamma$ be a graph on $\tilde S$ defined in Section~\ref{sec:Graph_Gamma}. Every framing $F\colon \tilde P\to\F$ defines the unique framing $(F^t,F^b)$ on $\tilde\Gamma$ as follows: for every vertex $v$ of $\tilde\Gamma$ define $F^t(v):=F(v^t)$ and $F^b(v):=F(v^b)$ (see Figure~\ref{vt-vb}). In this case we say that $(F^t,F^b)$ is \defin{induced} by $F$. We also denote $F^r(v):=F(v^r)$ and $F^l(v)=F(v^l)$.

Let $r$ be an internal edge of $\tilde\Gamma$ that separates two triangles $\tau$ and $\tau'$, and $v:=(\tau,r)$, $v':=(\tau',r)$. Then $F^t(v')=F^b(v)$, $F^b(v')=F^t(v)$, $F^r(v')=F^l(v)$, and $F^l(v')=F^r(v)$. Further, let $v$ and $v'$ two vertices of $\tilde\Gamma$ that lie in the same triangle and there is the edge $(v\to v')$ in $\tilde\Gamma$. Then $F^t(v')=F^b(v)$, $F^b(v')=F^r(v)$, $F^r(v')=F^t(v)$.

The following proposition is immediate:

\begin{prop}
Let $(F^t,F^b)$ be a transverse framing on $\tilde\Gamma$ such that
\begin{itemize}
    \item for every internal edge $r$ of $\tilde\Gamma$ that separates two triangles $\tau$ and $\tau'$, and $v:=(\tau,r)$, $v':=(\tau',r)$ holds $F^t(v')=F^b(v)$;
    \item for every two vertices $v$ and $v'$ of $\tilde\Gamma$ that lie in the same triangle and there is the edge $(v\to v')$ in $\tilde\Gamma$ holds $F^t(v')=F^b(v)$.
\end{itemize}
Then there exist a unique framing $F\colon \tilde P\to\F$ that induces $(F^t,F^b)$.
\end{prop}

\begin{rem} Let $(\rho, F)$ be a framed homomorphism. If $(F^t,F^b)$ is the framing induced by $F$, then $(F^t,F^b)$ is $\rho$-equivariant as well.
If $\mathfrak L=(\tilde\Gamma,T,F^t,F^b)$ is a $\pi_1(S)$-invariant transverse framed trivial $G$-local system on $\tilde\Gamma$, then $F$ is a framing of $\rho_\mathfrak L$, but in general, $\rho\neq \rho_\mathfrak L$.
\end{rem}

\begin{df}
We say that a transverse framed trivial $G$-local system $(\tilde\Gamma,T,F^t,F^b)$ is \defin{maximal} if the framing $(F^t,F^b)$ is induced by a maximal framing $F\colon \tilde P\to\F$.
\end{df}

\subsection{From representations to local systems}\label{Moves}

In this section, we show how to construct a maximal framed trivial $G$-local system on $\Gamma$ out of a maximal framed representation. 
Let $\mathcal T$ be an ideal triangulation of $S$ and $\tilde\Gamma$ be a graph on $\tilde S$ defined in Section~\ref{sec:Graph_Gamma}. Given a maximal framed homomorphism $(\rho,F)\in\Hom^{\fr}_+(\pi_1(S,b),G)$, then the framing $F$ gives rise to the $\rho$-equivariant transverse framing $(F^t,F^b)$ of $\tilde\Gamma$. We are going to define a local system $\mathfrak L=(\tilde\Gamma,T)$ such that $(F^t,F^b)$ will be its transverse framing and  $\rho_\mathfrak L=\rho$.



We will start with the two simplest cases when $S$ is a triangle and a triangulated quadrilateral, and then we describe a procedure for a general surface $S$.

\subsubsection{Triangle}\label{Turn}

Let $S=\tilde S$ be an ideal triangle with vertices $p_1,p_2,p_3$. The only triangulation $\T$ of $S$ consists of all external edges of $S$. Let $(\rho,F)$ be a maximal framed homomorphism. Since $\pi_1(S)$ is trivial, $\rho$ is trivial. Let $F(p_i)=F_i$ for all $i\in\{1,2,3\}$. Then $(F_1,F_2,F_3)$ is a maximal triple. Let the graph $\Gamma=\tilde\Gamma$ consist of three vertices $v$, $v'$ and $v''$ and three edges $(v'\gets v)$, $(v''\gets v')$ and $(v\gets v'')$.

Because of maximality by Proposition~\ref{act_pos_triples}, there exists an element $g\in G$ such that $g(F_1,F_2,F_3)=(\ell^+,\ell^-,\ell^1)$
(see Figure~\ref{Turn_figure}).

\begin{figure}[ht]
\begin{tikzpicture}
\draw (-3,3) node[label = ${g(F_1)=\ell^+}$] (v1) {} -- (-3,-2) node[label = below:${g(F_2)=\ell^-}$] (v2) {} -- (1,0.5) node[label = right:${g(F_3)=\ell_1}$] (v3) {} -- (-3,3);
\node[fill, circle, inner sep = 1.5pt, label=$v$](v) at (-2.5,0.5) {};
\node[fill, circle, inner sep = 1.5pt, label=$v''$](v'') at (-1.5,1.4) {};
\node[fill, circle, inner sep = 1.5pt, label= below:$v'$](v') at (-1.5,-0.3) {};

\draw[middlearrow={latex}] (v) to (v') ;
\draw[middlearrow={latex}] (v') to (v'');
\draw[middlearrow={latex}] (v'') to (v);

\end{tikzpicture}
\caption{Graph $\Gamma$ in a triangle}\label{Turn_figure}
\end{figure}
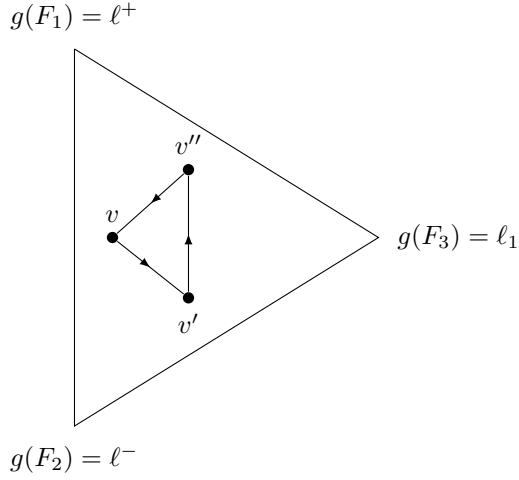

We define $T(v)\coloneqq g$ and $T(v'\gets v)\coloneqq T(v''\gets v')\coloneqq T(v\gets v'')\coloneqq \begin{pmatrix} -1 & 1 \\ -1 & 0 \end{pmatrix}$, then $T(v')=T(v'\gets v)T(v)$, $T(v'')=T(v''\gets v')T(v')$ and
$$T(v')F^t(v')=T(v')F_2=T(v'\gets v)\ell^-=\ell^+,$$
$$T(v')F^b(v')=T(v')F_3=T(v'\gets v)\ell^1=\ell^-,$$
$$T(v')F^r(v')=T(v')F_1=T(v'\gets v)\ell^+=\ell^1.$$ Similarly, one see that $T(v'')F^t(v'')=\ell^+$, $T(v'')F^b(v'')=\ell^-$ and $T(v'')F^r(v'')=\ell^1$. Moreover, $T(v\gets v'')T(v''\gets v')T(v'\gets v)=\Id$. So we have constructed a well defined maximal framed trivial $G$-local system in $\Gamma$. Notice, that local system is not unique, it depends on the choice of $g\in G$. This $g$ is uniquely defined up to the left multiplication by an element of $\K$ (see Proposition~\ref{act_pos_triples}).

\subsubsection{Quadrilateral. Crossing an edge of triangulation}\label{Quadruple}

Let $S=\tilde S$ be an ideal quadrilateral with cyclically ordered vertices $p_1,p_3,p_2,p_4$. Let $\T$ be the triangulation that consists of all external edges of $S$ and one internal edge connecting $p_1$ and $p_2$. Let $(\rho,F)$ be a maximal framed homomorphism. Since $\pi_1(S)$ is trivial, $\rho$ is trivial. Let $F(p_i)=F_i$ for all $i\in\{1,\dots,4\}$. Then $(F_1,F_4,F_2,F_3)$ is a positive quadruple. The graph $\Gamma=\tilde\Gamma$ consist of six vertices, four of them lie close to external edges and two $v$ and $v'$ lie close to the internal edge. So there exist an edge $(v'\gets v)$ or $(v\gets v')$. Without loss of generality, we assume the first case.

\begin{figure}[ht]
\scalebox{1}{
\begin{tikzpicture}

\draw (-3,3) node[label = \small${T(v)F_1=\ell^+}$] (v1) {} -- (-3,-2) node[label = below:\small${T(v)F_2=\ell^-}$] (v2) {} -- (0,0.5) node[label = left:\small${T(v)F_3=\ell^1}$] (v3) {} -- (-3,3);
\draw (-3,-2) -- (-6,0.5) node[label = right:\small{${T(v)F_4=\ell(a)}$}] (v4) {} -- (-3,3);

\draw (-3+7,3) node[label = \small${T(v')F_2=\ell^+}$] (v1) {} -- (-3+7,-2) node[label = below:\small${T(v')F_1=\ell^-}$] (v2) {} -- (0+7,0.5) node[label = left:\small${T(v')F_4=\ell^1}$] (v3) {} -- (-3+7,3);
\draw (-3+7,-2) -- (-6+7,0.5) node[label = right:\small${T(v')F_3=\ell(a)}$] (v4) {} -- (-3+7,3);

\node[fill, circle, inner sep = 1.5pt, label= $v$](a) at (-2.5,0.5) {};
\node[fill, circle, inner sep = 1.5pt, label= $v'$](b) at (-3.5,0.5) {};
\draw[middlearrow={latex}] (a) to (b);

\node[fill, circle, inner sep = 1.5pt, label= $v'$](b) at (-2.5+7,0.5) {};
\node[fill, circle, inner sep = 1.5pt, label= $v$](a) at (-3.5+7,0.5) {};
\draw[middlearrow={latex}] (a) to (b);

\draw[middlearrow={latex}] (-.25,1.5) to[bend left] (1.5,1.5) ;
\draw[middlearrow={latex}] (1.5,-0.5) to[bend left] (-.25,-0.5) ;

\node[label = $\begin{pmatrix}0 & \sqrt{a^{-1}} \\ -\sqrt{a} & 0\end{pmatrix}$] at (0.65,1.7) {};
\node[label = $\begin{pmatrix}0 & -\sqrt{a^{-1}} \\ \sqrt{a} & 0\end{pmatrix}$] at (0.65,-2) {};

\end{tikzpicture}}
\caption{Crossing an edge of triangulation}\label{Quadr_figure}
\end{figure}
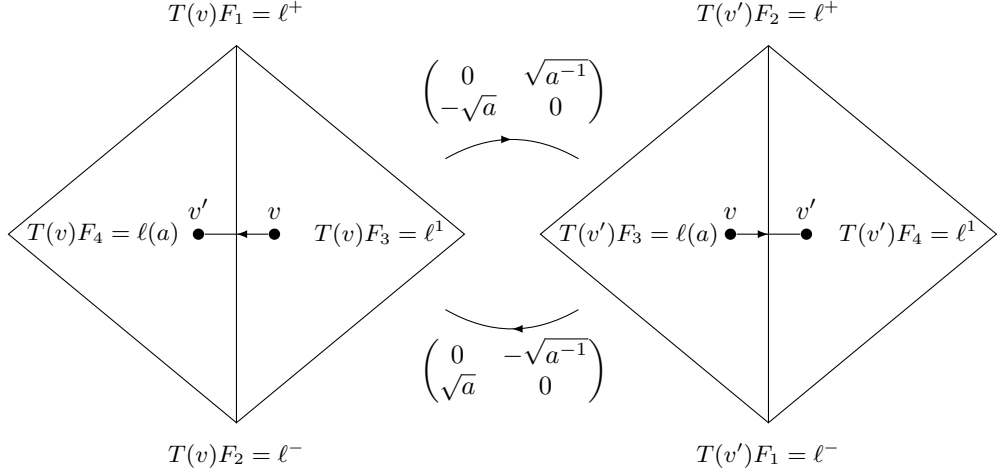

Because of transversality by Proposition~\ref{act_pos_quadr}, there exists an element $g\in G$ such that
$$g(F_1,F_4,F_2,F_3)=(\ell^+,\ell(a),\ell^-,\ell^1)$$
where $\ell(a)=(1,-a)^TA$ with $a\in A^\sigma_+$ (see Figure~\ref{Quadr_figure}). For every element of $a\in A^\sigma_+$ there exist the unique element $q\in A^\sigma_+$ such that $q^2=a$. The element $q$ is called the \defin{square root} of $a$ and denoted as $q=\sqrt{a}$. For a proof of the existence of the sqare root we refer to~\cite[Corollary~2.42]{ABRRW}. We define $T(v)\coloneqq g$, $T(v'\gets v)\coloneqq\begin{pmatrix}0 & \sqrt{a^{-1}} \\ -\sqrt{a} & 0\end{pmatrix}$, then $T(v')=T(v'\gets v)T(v)$. All other edges of $\Gamma$ lie inside one of two triangles of $S$, so we construct the map $T$ for them as in the previous subsection. So we have constructed a maximal framed trivial $G$-local system on $\Gamma$. The element $g$ is unique up to the left multiplication by an element of $\K$.


Let now $S$ be any punctured surface, $\mathcal T$ be an ideal triangulation of $S$ and $(\rho,F)\in\Hom^{\fr}_+(\pi_1(S),G)$. Let $\tilde\Gamma$ be the graph as in Section~\ref{sec:Graph_Gamma} with the framing $(F^t,F^b)$ induced by $F$. We can assume that the base point of the fundamental group agrees with some of point $b\in V_\Gamma$ that lies in a triangle $\tau$. We take $\tilde b\in V_{\tilde\Gamma}$ a lift of $b$.

Starting at $\tilde b$ and applying inductively the procedure described in Section~\ref{Turn} for every triangle and in Section~\ref{Quadruple} for every pair of adjacent triangles, we obtain a maximal framed local system on $\tilde\Gamma$. This local system have the property that the representation described in~(\ref{Graph_representation}) agrees with $\rho$. However, this local system is in general not $\pi_1(S)$-invariant, although the framing is $\rho$-equivariant. In the proof of Theorem~\ref{main_thm} in Section~\ref{section:parametrization}, we will modify this local system to make it $\pi_1(S)$-invariant.

\section{Parametrization of the space of maximal representations}\label{section:parametrization}

Now we are ready to prove one of the main results of this paper. As before, we assume $(A,\sigma)$ to be a Hermitian algebra. We denote $G:=\Sp_2(A,\sigma)$, and $\F:=\PP(\Is(\omega))$. Further, $S$ is a punctured surface with Euler characteristic $\chi(S)$ and with $p_e$ external punctures. The next theorem gives a parametrization of the space of maximal framed representations.

\begin{teo}\label{main_thm}
The space $\Rep_+^{\fr}(\pi_1(S),G)$ is homeomorphic to the following space:
$$((A^\sigma_+)^{p_e-3\chi(S)}\times U_{(A,\sigma)}^{1-\chi(S)})/U_{(A,\sigma)}$$
where $U_{(A,\sigma)}$ acts by conjugation in every factor.
\end{teo}

\begin{proof} Let $(\rho,F)$ be a maximal framed homomorphism. We pick some ideal triangulation $\T$ of $S$ with the set of internal edges $E_{in}$. The framing $F$ induces a framing $(F^t,F^b)$ of $\tilde\Gamma$.

Let $b\in S$ be the base point of $\pi_,(S,b)$ and $\tilde b\in\tilde S$ be a lift of $b$. As before, we assume $\tilde b\in V_{\tilde\Gamma}$. Let $g\in G$ be such that $g(F^t(\tilde b),F^b(\tilde b),F^r(\tilde b))=(\ell^+,\ell^-,\ell^1)$. We define
$$X\coloneqq\{(\rho,F,g)\mid (\rho,F)\in \Hom^{\fr}_+(\pi_1(S,b),G),\; g\in G : g(F^t(\tilde b),F^b(\tilde b),F^r(\tilde b))=(\ell^+,\ell^-,\ell^1)\}$$
and the equivalence relation $\sim$ on $X$ as follows: $(\rho,F,g)\sim (\rho',F',g')$ if and only if $(g\rho g^{-1},gF)=(g'\rho' g'^{-1},g'F')$.

Step 1: Assume $S=\tilde S$ is a triangulated polygon and $\Gamma$ defined as in Section~\ref{sec:Graph_Gamma}. Let $(\rho,F,g)\in X$. We define $T(b)=g$ and then apply the procedure defined in Section~\ref{Moves}. We obtain in this way a framed transverse $G$-local system on $\Gamma$. Let $v=(\tau,r)\in V_\Gamma$ for $r\in E_{in}$, then $v$ lies in the triangle $\tau$ with vertices $v^t,v^b,v^r\in P$. Let $\tau'$ be the triangle adjacent to $\tau$ along the edge $r=(v^t,v^b)$, and $v^b$, $v^t$ and $v^l$ be the vertices of $\tau'$. Let $b_v\in A^\sigma_+$ be defined as follows $T(v)F(v^l)=\ell(b_v)=(1,-b_v)^TA$. The elements $b_v$ are uniquely defined for all $v\in V_\Gamma$. Moreover by construction in Section~\ref{Moves}, $b_v=b_{v'}$ where $v'=(\tau',r)$. Therefore, we can assign the element $b_v$ not to the vertex of $\Gamma$ but to the internal edge $r$ of $\mathcal T$. So we will write $b_r$ instead.

Therefore, we obtain the following map:
$$\begin{matrix}
\mathcal X\colon & X & \to & (A^\sigma_+)^{E_{in}}\\
 & (\rho,F,g) & \mapsto & (b_r)_{r\in E_{in}}.
\end{matrix}$$
Notice that $\mathcal X(\rho,F,g)=\mathcal X(\rho',F',g')$ if and only if  $(\rho,F,g)\sim(\rho',F',g')$. This map is surjective because any choice of $(b_e)_{e\in E_{in}}$ defines a transverse framed $G$-local system on $\Gamma$ that by Section~\ref{Loc_to_Rep} defines a framed homomorphism $(\rho,F)\in\Hom^{\fr}_\T(\pi_1(S),G)$.

We also have the $\K$-action on $X$ by the left multiplication in the last component and by componentwise conjugation on $(A^\sigma_+)^{E_{in}}$ that is compatible with $\mathcal X$, i.e. let $\diag(u,u)=k\in\K$ for $u\in U_{(A,\sigma)}$ then $\mathcal X(\rho,D,kg)=(ub_ru^{-1})_{r\in E_{in}}$.
We obtain the quotient map:
$$\bar{\mathcal X}\colon \Rep^{\fr}_+(\pi_1(S,b),G)\to (A^\sigma_+)^{E_{in}}/\K$$
which is a homeomorphism.

Step 2: Now let us turn to the general case. We choose a connected fundamental domain $S_0$ for $S$ in $\tilde S$ that completely consists of ideal triangles and $\tilde b$ lies in one of them. The fundamental domain $S_0$ is an ideal polygon, so we can apply the procedure above and obtain a transverse framed $G$-local system over $\Gamma_0$ which is a subgraph of $\tilde \Gamma$ contained in $S_0$. We want to extend this local system to the entire graph $\tilde\Gamma$ in a $\pi_1(S)$-invariant way.

To obtain $\pi_1(S)$-invariant local system, according to Section~\ref{Loc_to_Rep}, we need to define $T(\tilde v)$ for every $\tilde v\in V_\Gamma$ as follows: let $\gamma\in\pi_1(S)$ be the unique element such that $\gamma v\in S_0$. We define $T(\tilde v)\coloneqq T(\gamma \tilde v)\rho(\gamma)$. From this follows immediately that for every $\tilde v\in V_{\tilde\Gamma}$ and for every $\gamma\in\pi_1(S)$, $T(\gamma\tilde v)=T(\tilde v)\rho(\gamma)^{-1}$.

The local system $(\tilde\Gamma,T)$ is indeed $\pi_1(S)$-invariant because: let $(w\gets v)\in E_{\tilde\Gamma}$ and $\gamma\in\pi_1(S)$. Then $$T(\gamma\tilde w\gets\gamma\tilde v)=T(\gamma\tilde w)T(\gamma\tilde v)^{-1}=T(\tilde w)\rho(\gamma)^{-1}\rho(\gamma)T(\tilde v)^{-1}=T(\tilde w)T(\tilde v)^{-1}=T(\tilde w\gets\tilde v).$$

The framing $(F^t,F^b)$ of $(\tilde\Gamma,T)$ is adapted to $(\tilde\Gamma, T)$. Indeed, if $\tilde v$ is contained in $S_0$, then by construction $T(\tilde v)(F^t(\tilde v),F^b(\tilde v))=(\ell^+,\ell^-)$. If $\tilde v\in V_{\tilde\Gamma}$ any vertex, then there exist the unique $\gamma\in\pi_1(S)$ with $\gamma\tilde v$ is contained in $S_0$. Then
$$T(\tilde v)(F^t(\tilde v),F^b(\tilde v))=T(\gamma \tilde v)\rho(\gamma)\rho(\gamma)^{-1}(F^t(\gamma\tilde v),F^b(\gamma\tilde v))=T(\gamma \tilde v)(F^t(\gamma\tilde v),F^b(\gamma\tilde v))=(\ell^+,\ell^-).$$

The framed local system $(\tilde \Gamma,T, F^t,F^b)$ is determined by its restriction to $\Gamma_0$ and the maps $T(\tilde w\gets \tilde v)$ for $\tilde v\in S_0$ and $\tilde w\notin S_0$. So we need to understand the maps $T(\tilde w\gets \tilde v)$.

Let $e,e'$ be edges of the triangulation $\tilde\T$ that are in the boundary of $S_0$ but not for $\tilde S$ and let $\gamma\in\pi_1(S)$ be the unique element such that $\gamma(e)=e'$. Without loss of generality assume $(\tilde w\gets \tilde v)\in E^+_{\tilde\Gamma}$ that crosses $e$ and $\tilde v\in S_0$, $\tilde w\notin S_0$. Then the edge $(\gamma_i\tilde w\gets \gamma_i\tilde v)\in E^+_{\tilde\Gamma}$ crosses $e'$ and $\gamma_i\tilde v\notin S_0$, $\gamma_i\tilde w\in S_0$.
$$T(\tilde w\gets \tilde v)= T(\tilde w)T(\tilde v)^{-1}=T(\gamma \tilde w)\rho(\gamma)T(\tilde v)^{-1}.$$
Therefore:
$$T(\tilde w\gets \tilde v)(\ell^+,\ell^-)
=T(\gamma \tilde w)\rho(\gamma)T(\tilde v)^{-1}(\ell^+,\ell^-)
=T(\gamma \tilde w)\rho(\gamma)(F^t(\tilde v),F^b(\tilde v))$$
$$=T(\gamma \tilde w)(F^t(\gamma\tilde v),F^b(\gamma\tilde v))
=T(\gamma \tilde w)(F^b(\gamma\tilde w),F^t(\gamma\tilde w))
=(\ell^-,\ell^+).$$
This means that $T(\tilde w\gets \tilde v)=l_e\omega$ where $l_e\in\L$ is the uniquely defined element and $\omega=\Om$. So for every pair $\{e,e'\}$ of external edges that are related by a nontrivial element of $\pi_1(S)$ we get an element of $\L$. By the polar decomposition for Hermitian algebras~\cite[Theorem~2.66]{ABRRW}, the element $l_e$ can be written in the unique way as $l_e=\diag(u_eb_e,u_eb_e^{-1})$, where $u_e\in U_{(A,\sigma)}$, $b_e\in A^\sigma_+$. We denote the set of all $\{e,e'\}$ pairs as above by $E_0$. Its cardinality is $1-\chi(S)$.

As in Step 1, we obtain a map
$$\begin{matrix}
\mathcal X\colon & X & \to & (A^\sigma_+)^{E_{in}}\times U_{(A,\sigma)}^{E_0}\\
 & (\rho,F,g) & \mapsto & (b_r,u_e)_{r\in E_{in},\;e\in E_0}.
\end{matrix}$$
Notice that $\mathcal X(\rho,F,g)=\mathcal X(\rho',F',g')$ if and only if  $(\rho,F,g)\sim(\rho',F',g')$. Further, this map is surjective because any choice of $(b_r,u_e)_{r\in E_{in},\;e\in E_0}$ defines a $\pi_1(S)$-invariant transverse framed $G$-local system on $\tilde\Gamma$ that by~\ref{Loc_to_Rep} defines a framed homomorphism $(\rho,F)\in\Hom^{\fr}_+(\pi_1(S),G)$.

We also have the $\K$-action on $X$ by the left multiplication in the last component and by componentwise conjugation on $(A^\sigma_+)^{E_{in}}\times U_{(A,\sigma)}^{E_0}$ that is compatible with $\mathcal X$, i.e. $\mathcal X(\rho,D,kg)=(ub_ru^{-1},uu_eu^{-1})_{r\in E_{in},\;e\in E_0}$ where $k=\diag(u,u)$.
We obtain the quotient map:
$$\bar{\mathcal X}\colon \Rep^{\fr}_+(\pi_1(S,b),G)\to ((A^\sigma_+)^{E_{in}}\times U_{(A,\sigma)}^{E_0})/U_{(A,\sigma)}$$
which is a homeomorphism.
\end{proof}

\begin{cor}
The space $\Rep_+^{\fr}(\pi_1(S),G)$ is homotopy equivalent to $\K^{1-\chi(S)}/\K$. Let $k$ be the number of connected components of $\K$, then $\Rep_+^{\fr}(\pi_1(S),G)$ has $k^{1-\chi(S)}$ connected components.
\end{cor}

For the proof of this result we refer to~\cite[Section~7]{AGRW} where it is proven for $G=\Sp(2n,\R)$. The proof for $G=\Sp_2(A,\sigma)$ with $(A,\sigma)$ Hermitian is essentially the same. The idea of the proof is to use the fact that $A^\sigma_+$ is a cone, so it is contractible. In~\cite{AGRW} is shown that there is a contraction that is invariant under the action of $U_{(A,\sigma)}$ on $A^\sigma_+$ by conjugation.

\subsection{Connected components of the space of maximal representations}\label{sec:connected_components}

In this section, we study the space of maximal non-framed representations. We assume that $S$ is a punctured surface without boundary and $G=\Sp_2(A,\sigma)$ where $(A,\sigma)$ is a Hermitian algebra. Using the results for framed representations, one can show that the spaces of maximal framed and non-framed maximal representations has the same number of connected components. Further we provide examples and count connected components of spaces of maximal representations for classical Hermitian Lie groups of tube type.

Let $\Rep^{\fr}(\pi_1(S),G)\to \Rep(\pi_1(S),G)$ be the natural projection. We call this projection the \defin{framing forgetting map}. As we have seen in Proposition~\ref{max.prop}, $\Rep_+(\pi_1(S),G)$ is the image of $\Rep_+^{\fr}(\pi_1(S),G)$ under the framing forgetting map.

\begin{teo}
The framing forgetting map $\Rep_+^{\fr}(\pi_1(S),G)\to\Rep_+(\pi_1(S),G)$ induces a bijection between $\pi_0(\Rep_+^{\fr}(\pi_1(S),G))$ and $\pi_0(\Rep_+(\pi_1(S),G))$.
\end{teo}

This theorem is proven in~\cite[Chapter~7]{AGRW} for maximal symplectic representations and in~\cite{GRW} for positive representations. The key ingredient of this proof is the path lifting property of the framing forgetting map restricted to $\Rep_+^{\fr}(\pi_1(S),G)$. One uses it to show that if two points are in the same connected component in $\Rep_+(\pi_1(S),G)$ then their preimages are in the same connected component in  $\Rep_+^{\fr}(\pi_1(S),G)$.

Now we apply this theorem for classical Hermitian Lie group that can be seen as $\Sp_2(A,\sigma)$.

\begin{itemize}
    \item Let $A=\Mat(n,\R)$ and $\sigma$ be the transposition. Then $A^\sigma_+=\Sym^+(n,\R)$, $U_{(A,\sigma)}=\OO(n)$ and $\Sp_2(A,\sigma)=\Sp(2n,\R)$ (see Example~(1) in Section~\ref{ex:groups}). We obtain that $\Rep_+^{\fr}(\pi_1(S),\Sp(2n,\R))$ is homeomorphic to
        $$(\Sym^+(n,\R)^{-3\chi(S)}\times\OO(n)^{1-\chi(S)})/\OO(n)$$
        Since $\OO(n)$ has two connected components, the number of connected components of $\Rep_+^{\fr}(\pi_1(S),\Sp(2n,\R))$ and of $\Rep_+(\pi_1(S),\Sp(2n,\R))$ is $2^{1-\chi(S)}$.
    \item Let $A=\Mat(n,\CC)$ and $\sigma$ be the transposition composed with the complex conjugation. Then $A^\sigma_+=\Herm^+(n,\CC)$, $U_{(A,\sigma)}=\UU(n)$ and $\Sp_2(A,\sigma)=\UU(n,n)$ (see Example~(2) in Section~\ref{ex:groups}). We obtain that $\Rep_+^{\fr}(\pi_1(S),\UU(n,n))$ is homeomorphic to
        $$(\Herm^+(n,\CC)^{-3\chi(S)}\times\UU(n)^{1-\chi(S)})/\UU(n)$$
        Since $\UU(n)$ is connected, the spaces $\Rep_+^{\fr}(\pi_1(S),\UU(n,n))$ and $\Rep_+(\pi_1(S),\UU(n,n))$ are connected.
    \item Let $A=\Mat(n,\HH)$ and $\sigma$ be the transposition composed with the quaternionic conjugation. Then $A^\sigma_+=\Herm^+(n,\HH)$, $U_{(A,\sigma)}=\Sp(n)$ and $\Sp_2(A,\sigma)=\SO^*(4n)$ (see Example~(3) in Section~\ref{ex:groups}). We obtain that $\Rep_+^{\fr}(\pi_1(S),\SO^*(4n))$ is homeomorphic to
        $$(\Herm^+(n,\HH)^{-3\chi(S)}\times\Sp(n)^{1-\chi(S)})/\Sp(n)$$
        Since $\Sp(n)$ is connected, the spaces $\Rep_+^{\fr}(\pi_1(S),\SO^*(4n))$ and $\Rep_+(\pi_1(S),\SO^*(4n))$ are connected.
\end{itemize}

We complete this description of spaces of maximal representations with results that cannot by obtained using the technics provided in this paper. For more details we refer to~\cite{GRW}.

\begin{itemize}
    \item If $G=\SO_0(2,n)$ the identity component of the indefinite orthogonal group $\SO(2,n)$, then the space $\Rep_+^{\fr}(\pi_1(S),\SO_0(2,n))$ is homeomorphic to
        $$(\Omega^{-3\chi(S)}\times \OO(n-1)^{1-\chi(S)})/\OO(n-1),$$
        where $\Omega$ is some open proper convex cone in $\R^n$. Since the group $\OO(n-1)$ has two connected components, the number of connected components of $\Rep_+^{\fr}(\pi_1(S),\SO_0(2,n))$ and $\Rep_+(\pi_1(S),\SO_0(2,n))$ is $2^{1-\chi(S)}$.
    \item If $G=E_{7(-25)}$, then the space $\Rep_+^{\fr}(\pi_1(S),E_{7(-25)})$ is homeomorphic to
        $$(\Herm^+(3,\Oc)^{-3\chi(S)}\times F_4^{1-\chi(S)})/F_4,$$
        where $\Herm^+(3,\Oc)$ is the space of octonionic Hermitian $3\times 3$-matrices, $F_4$ is the compact from of the complex exceptional group $F_4$. Since the group $F_4$ is connected, the spaces $\Rep_+^{\fr}(\pi_1(S),E_{7(-25)})$ and $\Rep_+(\pi_1(S),E_{7(-25)})$ are connected.
\end{itemize}

\newpage
\bibliographystyle{plain}
\bibliography{bibl}

\end{document}